\documentclass{article}

\usepackage[english]{babel}

\usepackage{amsmath,amssymb,setspace}
\usepackage{amsthm}
\usepackage[english]{babel}
\usepackage{amsfonts}
\usepackage{calligra}
\usepackage[T1]{fontenc}
\usepackage{xstring}
\usepackage{graphics}
\usepackage[numbers,sort&compress]{natbib}
\usepackage{multicol}
\date{}
\usepackage{longtable}
\usepackage{float}
\usepackage[margin=0.7 in]{geometry}
\usepackage{graphicx}
\usepackage{mathtools}
\usepackage{amsmath}
\usepackage{subcaption}
\usepackage{amsmath, amssymb, setspace}

\usepackage{amsthm}
\usepackage[english]{babel}
\usepackage{amsfonts}
\usepackage{calligra}
\usepackage[T1]{fontenc}
\usepackage{xstring}

\usepackage{amssymb}
\usepackage{amsthm}
\usepackage{amsmath}
\usepackage{lipsum}
\newtheorem{definition}{Definition}
\theoremstyle{plain}
\theoremstyle{definition}
\theoremstyle{remark}
\newtheorem{theorem}{Theorem}

\newtheorem{example}{Example}

\title{On Homogeneous System of Fractal Differential Equations }
\author{Alireza Khalili Golmankhaneh$^1$, Donatella Bongiorno$^2$,\\
$^1$ Department of Physics, Urmia Branch,\\ Islamic Azad University, Urmia 63896,West Azerbaijan,  Iran\\
alirezakhalili2002@yahoo.co.in\\
$^2$ Department of Engineering, University of Palermo,
Palermo 90100, Italy\\
donatella.bongiorno@unipa.it\\
}

\begin{document}

\maketitle

\let\thefootnote\relax
\footnote{Corresponding author Alireza Khalili Golmankhaneh}
\begin{abstract}
In this paper, a homogeneous system of n $\alpha$-order linear fractal differential equation is defined and the set of its fundamental solutions through the corresponding Wronskian matrix is described.  Finally,  the solutions of some assigned autonomous homogeneous systems are plotted to show the details previously proved.
\end{abstract} 

\section{Introduction}
Fractals are geometric shapes characterized by having dimensions that are larger and non-integer compared to their topological dimensions\cite{falconer1999techniques,jorgensen2006analysis}.

Lightning does not follow a linear path, clouds lack perfect spherical shapes, and mountains do not conform to cone-like structures. The complexity of natural shapes differs fundamentally from those in traditional geometry. Specifically, fractal geometry captures the intricate nature of these shapes, representing a qualitative difference rather than just a matter of degree \cite{Mandelbro}.

 The research was conducted to examine the process of identifying and analyzing fractal patterns, particularly those emerging from social processes \cite{brown2010fractal}. Several techniques for quantifying the spatial and temporal attributes of fractal patterns were presented \cite{moreira2021fractal}. It covers techniques like box-counting, which can determine fractal dimensions in diverse contexts such as Saturn's rings, forest landscapes, brain imaging data, and text analysis \cite{moreira2021fractal}. The analytical applications of IFS methods, demonstrating their direct uses as well as novel analytical approaches inspired by the IFS fractal framework, were showcased \cite{kunze2011fractal}.

The systematic analysis, modeling, and synthesis of integrated fractal and fractal-rate point processes are explored, presenting a rigorous yet practical approach \cite{lowen2005fractal}. These processes amalgamate the scaling properties of fractals with the discrete nature of random point processes \cite{lowen2005fractal}. The comprehensive understanding of polymer synthesis and curing, by establishing structural and physical foundations through fractal analysis, was thoroughly investigated \cite{kozlov2013fractal}. Theoretical concepts and practical applications of fractals and multifractals, as discussed in \cite{ghanbarian2017fractals},  cater to various scientific communities, including those involved in petroleum, chemical, civil and environmental engineering, atmospheric research, and hydrology \cite{feder2013fractals}. Fractal analysis has been utilized in various fields, including the identification of coding regions in DNA and the assessment of space-filling properties in tumors, blood vessels, and neurons. Additionally, fractal concepts have been effectively integrated into models of biological phenomena, such as epithelial cell growth, blood vessel development, periodontal disease progression, and viral infections \cite{cross1997fractals}. Ecologists have harnessed the power of fractals to tackle fundamental questions pertaining to scale, measurement, and hierarchy within ecological systems \cite{sugihara1990applications}. Self-similar and self-affine processes are ubiquitous in nature, manifesting across various scales and contexts, including galaxies, landscapes, earthquakes, geological formations, aggregates, colloids, rough surfaces, interfaces, glassy materials, polymers, proteins, and other large molecules \cite{bunde2013fractals}. The explanations covered transformations in metric spaces, dynamics within fractals, fractal dimension, interpolation techniques on fractals, Julia sets, parameter spaces, and measures associated with fractals \cite{Barnsley2014p,attia2021note}. Fractal dimensions were initially introduced using random walks and Sierpinski gaskets. The exploration culminated with diffusion-limited aggregates, closely associated with critical exponents \cite{stauffer2017fractals}. Self-similarity and scale invariance, concepts rooted in both phase transitions and fractal geometry, intersect in numerous domains of physics. The discussion extends to advanced theoretical subjects such as chaos theory, percolation phenomena, turbulence models, and renormalization group methods \cite{pietronero2012fractals,takayasu1990fractals}. Constitutive equations for fractal elasticity are provided, establishing a linear elastic relationship based on the definition of fractal elastic potential. The physical dimensions of the second derivatives of the elastic potential are determined by the fractal dimensions of both stress and strain \cite{carpinteri2004fractal}. Recent physics research has introduced fractal time, characterized by self-similar properties and fractional dimension \cite{Welch-5,Vrobel}.\\
The significance of fractals has spurred researchers to analyze them using a myriad of methods, including fractional calculus \cite{Trifcebook}, harmonic analysis \cite{kigami2001analysis}, measure theory \cite{giona1995fractal, freiberg2002harmonic,bongiorno2018derivatives, jiang1998some, bongiorno2015fundamental}, nonstandard methodologies \cite{teh2019resolution,nottale2011scale}, probabilistic approaches \cite{Barlow} and fractional space \cite{stillinger1977axiomatic}.
The analysis was conducted on an integral of the s-Riemann type, where the gauge is a positive constant but the points involved in the s-Riemann sums were not randomly chosen \cite{BDonattalerr}.
 The investigation into numerical methods for solving partial differential equations on fractals was conducted \cite{contreras2022finite}. This study delves into both strong and weak forms of the equations, utilizing standard graph Laplacian matrices and discrete approximations of fractal sets.\\
While the Riemann-like method remains significant as an extension of classical calculus, fractal calculus offers a mathematical framework tailored for handling equations with solutions displaying fractal characteristics, including fractal sets and curves \cite{parvate2009calculus, parvate2011calculus}. Fractal calculus stands out for its elegant and algorithmic approaches, which are deemed superior to other methods, rendering it particularly attractive \cite{Alireza-book}.\\ The fractal local Mellin transform and fractal non-local transform have been employed as tools for solving fractal differential equations \cite{KhaliliGolmankhanehWelchSerpaJørgensen+2023}. The concept of Laplace transform and local Fourier transform has been extended to fractal curves, facilitating the solution of fractal differential equations with constant coefficients \cite{khalili2023fractalLapace,Fourier1ttttttt}. This study investigates fractal time within the framework of economic models, employing both local and non-local fractal Caputo derivatives \cite{khalili2021economic}. The gauge integral approach has been successfully employed in the generalization of $F^{\alpha}$-calculus (FC). This generalization focuses on integrating functions over a specific subset \cite{golmankhaneh2016fractal,golmankhaneh2023fuzzification} of the real number line, encompassing singularities that arise within fractal sets. The concept establishes a connection between the order of non-local fractal derivatives and the fractal Hurst exponent \cite{golmankhaneh2021fractalBro}. Fractal differential equations have been solved, and their stability conditions have been established \cite{tuncc2020stability}. The Einstein field equations and their analogues were presented, emphasizing the relevance and practicality of fractal geometry \cite{golmankhaneh2023einstein}. Fractal integro-differential equations for different types of circuits exposed to zero-mean additive white Gaussian noise on fractal sets have been formulated, integrating a fractal time component \cite{banchuin2022noise,banchuin20224noise,Rewid3}. The suggestion was made that two distinct del-operators, each acting on a vector field and a scalar field, could be defined in spaces with non-integer dimensions. These del-operators were then utilized to derive the conventional formulation of the Laplacian and the fundamental vector differential operators in fractional-dimensional space \cite{balankin2023vector}. Fractal functional differential equations were formulated as a framework that offers a mathematical model for phenomena characterized by fractal time and fractal structure \cite{golmankhaneh2023initial}. The study demonstrated that non-local fractal derivatives can effectively characterize fractional Brownian motion on thin Cantor-like sets. Additionally, it introduced a staircase function associated with a fractal comb, which was utilized to define derivatives and integrals for functions defined on these combs \cite{golmankhaneh2023fractalrede}. Fractal integral and differential forms were defined utilizing nonstandard analysis \cite{khalili2023non}. Diffusion and fractional Brownian motion in fractal-structured media were categorized, and fractal stochastic differential equations were formulated \cite{golmankhaneh2021equilibrium,khalili2019fractalcat,khalili2019random}. The study \cite{golmankhaneh2023solving} investigated the analogues of the separable method and integrating factor methodology for solving $\alpha$-order differential equations. Inspiration was drawn from Newton, Lagrange, Hamilton, and Appell for the proposed fractal analogue of mechanics. Furthermore, to establish the Langevin equation on fractal curves, fractal velocity and acceleration were defined \cite{golmankhaneh2023classical}. The conventional Fokker-Planck Equation and its fractal version, incorporating fractal derivatives, were linked to each other in the study \cite{megias2023dynamics}. The study explored various mathematical models for the growth of double-sized cancer in the fractal temporal dimension \cite{golmankhaneh2024modeling}. Additionally, a fractal discharging model for batteries was developed to investigate the impact of non-locality on solution behavior and how the system's previous state influences its current state \cite{AliGolmankhanehYilmazer}.
\\
In this investigation, we analyze systems of fractal differential equations and explore their solutions. \\
The paper is structured as follows:\\
In Section \ref{1g}, we offer a concise overview of fractal calculus. Section \ref{2g} delves into the study of systems of $\alpha$-order linear fractal differential equations, accompanied by the proof of related theorems. In Section \ref{3g}, we focus on solving homogeneous fractal linear systems with constant coefficients. Finally, Section \ref{4g} presents the conclusion.
\section{Preliminaries \label{1g}}
In this section,  a concise overview  of fractal calculus in the context of fractal sets (see \cite{parvate2009calculus,parvate2011calculus,Alireza-book}) and the definition of fractal matrix with its fractal derivative and its fractal integral are given.
\begin{definition}
A flag function of a set $F \subset \mathbb{R}$ and a closed interval $I$ is defied as follows:
\begin{equation}
  \rho(F,I)=
  \begin{cases}
    1, & \text{if } F\cap I\neq\emptyset;\\
    0, & \text{otherwise}.
  \end{cases}
\end{equation}
\end{definition}

\begin{definition}
Given $\delta > 0$, a constant, and given $P_{[a,b]}$ a subdivision of $[a,b]$, the coarse-grained mass of a fractal set $F \subset \mathbb{R} $ is defined as follows:

\begin{equation}
  \gamma_{\delta}^{\alpha}(F,a,b)=\inf_{|P|\leq
\delta}\sum_{i=0}^{n-1}\Gamma(\alpha+1)(t_{i+1}-t_{i})^{\alpha}
\rho(F,[t_{i},t_{i+1}]),
\end{equation}
where $|P|=\max_{0\leq i\leq n-1}(t_{i+1}-t_{i})$, and $0< \alpha\leq1$.
\end{definition}

\begin{definition}
 The mass function of a fractal set $F \subset \mathbb{R}$ is the limit as $\delta$ tends to zero of the coarse-grained mass:
\begin{equation}
  \gamma^{\alpha}(F,a,b)=\lim_{\delta\rightarrow0}\gamma_{\delta}^{\alpha}(F,a,b).
\end{equation}
\end{definition}

\begin{definition}
The $\gamma$-dimension of the intersection between $F \subset \mathbb{R}$ and  the interval $[a, b]$ is defined as follows:
\begin{align}
  \dim_{\gamma}(F\cap
[a,b])&=\inf\{\alpha:\gamma^{\alpha}(F,a,b)=0\}\nonumber\\&
=\sup\{\alpha:\gamma^{\alpha}(F,a,b)=\infty\}
\end{align}
\end{definition}

\begin{definition}
For a fractal set $F \subset \mathbb{R}$, the integral staircase function of order $\alpha$ is defined as follows:
\begin{equation}
 S_{F}^{\alpha}(x)=
 \begin{cases}
   \gamma^{\alpha}(F,a_{0},x), & \text{if } x\geq a_{0}; \\
   - \gamma^{\alpha}(F,x,a_{0}), & \text{otherwise}.
 \end{cases}
\end{equation}
where $a_{0}$ is an arbitrary fixed real number.
\end{definition}
\begin{definition}
 Let $F$ be a fractal subset of the real line. A function $f: F \rightarrow \mathbb{R}$ is said to be $F$-continuous at $x \in F$ if
\begin{equation}
f(x) = \underset{ y\rightarrow
x}{F_{-}\text{lim}}f(y).
\end{equation}
whenever the $F_{-}\text{lim}$ exists.
\end{definition}

\begin{definition}
Let $f$ be a function defined on a $\alpha$-perfect fractal set $F \subset \mathbb{R}$, and let  $x$ be a point of $F \subset \mathbb{R}$. The $F^{\alpha}$-derivative of $f$ at the point $x$ is:
\begin{equation}
  D_{F}^{\alpha}f(x)=
  \begin{cases}
    \underset{ y\rightarrow
x}{F_{-}\text{lim}}~\frac{f(y)-f(x)}{S_{F}^{\alpha}(y)-S_{F}^{\alpha}(x)}, & \text{if } x\in F; \\
    0, & \text{otherwise}.
  \end{cases}
\end{equation}
if the fractal limit exists \cite{parvate2009calculus}.
\end{definition}

\begin{definition}
Let I=[a,b]. Let $F \subset \mathbb{R}$ be an $\alpha$-perfect fractal subset of $[a,b]$ and let $x\in F.$ The $F^{\alpha}$-integral of a bounded function $f$ defined on $F \subset \mathbb{R}$ where $S^{\alpha}_F(*)$ is finite on $I$, is defined as follows:
\begin{align}
  \int_{a}^{b}f(x)d_{F}^{\alpha}x&=\sup_{P_{[a,b]}}
\sum_{i=0}^{n-1}\inf_{x\in F\cap
I}f(x)(S_{F}^{\alpha}(x_{i+1})-S_{F}^{\alpha}(x_{i}))
\nonumber\\&=\inf_{P_{[a,b]}}
\sum_{i=0}^{n-1}\sup_{x\in F\cap
I}f(x)(S_{F}^{\alpha}(x_{i+1})-S_{F}^{\alpha}(x_{i})).
\end{align}
\end{definition}
\begin{definition}
Let $F \subset \mathbb{R}$ and let $a_{ij}:F\rightarrow R$ where $i=1,2,...,m$ and $j=1,2,...,n$ be fractal functions. Therefore the fractal matrix $m\times n$  is defined as follows
\begin{equation}\label{frdq2}
       \mathbf{A}(t)=\left(
         \begin{array}{ccc}
           a_{11}(t) & \cdots & a_{1n}(t)  \\
           \vdots &  & \vdots \\
           a_{m1}(t) & \cdots & a_{mn}(t) \\
         \end{array}
       \right).
\end{equation}
Moreover, whenever $n=1$ the fractal matrix will be called fractal vector and we will write:
\begin{equation}\label{olpm89}
  \mathbf{x}(t)=\left(
         \begin{array}{c}
           x_{1}(t) \\
           \vdots \\
           x_{m}(t) \\
         \end{array}
       \right)
\end{equation}
  If each element of $\mathbf{A}(t)$  is a $F$-continuous function, then the fractal matrix $\mathbf{A}(t)$ is a $F$-continuous. Moreover, if  each element of $\mathbf{A}(t)$ is a $F^{\alpha}$-differentiable function, then $\mathbf{A}(t)$ is $F^{\alpha}$-differentiable and the fractal derivative of $\mathbf{A}(t)$ is the following:
\begin{equation}\label{dsww12}
  D_{F}^{\alpha}\mathbf{A}(t)=\left(D_{F}^{\alpha}a_{ij}(t)\right).
\end{equation}
Finally, whenever $F\subset [a,b]$,  the fractal integral of a fractal matrix $\mathbf{A}(t)$, is given as follows:
\begin{equation}\label{rfeeewsaq}
 \int_{a}^{b}\mathbf{A}(t)d_{F}^{\alpha}t=
 \left(\int_{a}^{b}a_{ij}(t)d_{F}^{\alpha}t\right).
\end{equation}
\end{definition}

\section{Homogeneous System of $\alpha$-Order  Linear Fractal Differential Equation \label{2g}}
In this section, we introduce a homogeneous system of $\alpha$-order linear fractal differential equations and through the corresponding Wronskian matrix, we discuss some characteristic properties of the set of its solutions.
The general form of a homogeneous system of $n$ $\alpha$-order fractal differential linear equations is as follows:
\begin{equation}\label{ijnuhygb}
\begin{split}
  D_{F}^{\alpha}x_{1}(t)&=p_{11}(t)x_{1}(t)+...+p_{1n}(t)x_{n}(t)\\
  \vdots&\\
  D_{F}^{\alpha}x_{n}(t)&=p_{n1}(t)x_{1}(t)+...+p_{nn}(t)x_{n}(t)
  \end{split}
\end{equation}
In short form, Eq. \eqref{ijnuhygb} can be expressed as:
\begin{equation}\label{ijnbhy}
  D_{F}^{\alpha}\mathbf{x}(t)=\mathbf{P}(t)\mathbf{x}(t),~~\forall t\in F,
\end{equation}
where
\begin{equation}\label{ijko951}
  \mathbf{x}(t)=\left(
               \begin{array}{c}
                 x_{1}(t) \\
                 \vdots \\
                 x_{n}(t) \\
               \end{array}
             \right)~~~~~and~~~\mathbf{P}(t)=\left(
         \begin{array}{ccc}
           p_{11}(t) & \cdots & p_{1n}(t)  \\
           \vdots &  & \vdots \\
           p_{n1}(t) & \cdots & p_{nn}(t) \\
         \end{array}
       \right).
\end{equation}
We assume that each element of the fractal matrix  $\mathbf{P}(t)$ is a real value $F$-continuous function on $F$.  Moreover,  we will use the following notation:
\begin{equation}\label{dddswwswq}
  \mathbf{x}^{1}(t)=\left(
      \begin{array}{c}
        x_{11}(t) \\
        x_{21}(t) \\
        \vdots\\
         x_{n1}(t) \\
      \end{array}
    \right),\cdots, \mathbf{x}^{k}(t)=\left(\begin{array}{c}
        x_{1k}(t) \\
        x_{2k}(t) \\
        \vdots\\
         x_{nk}(t) \\
      \end{array}
    \right),\cdots
\end{equation}
where $x_{ij}(t)=x_{i}^{(j)}(t)$ denotes the $i^{th}$ component of the $j^{th}$ solution $\mathbf{x}^{(j)}(t)$ of Eq.\eqref{ijnbhy}. We say that the solution of a given system of $\alpha$-order linear fractal differential equation is the collection of all $F^{\alpha}$-differentiable complex vector functions defined on $F$: $\mathbf{x}^{(j)}(t)$ such that, by substituting each  $\mathbf{x}^{(j)}(t)$  for the individual $x_{ij}(t)$, the system \eqref{ijnbhy}  is satisfied. Now let us observe that by the linearity of the system, if the fractal vector functions $\mathbf{x}^{(1)}(t),~\mathbf{x}^{(2)},...,\mathbf{x}^{(n)}(t)$  are solutions of the system \eqref{ijnbhy}, then the linear combination $c_{1}\mathbf{x}^{(1)}(t)+c_{2}\mathbf{x}^{(2)}(t)+...+c_{n}\mathbf{x}^{(n)}(t) $ is also a solution for any constants $c_{1},c_{2},...,c_{n}$, so the set of all the solutions of Eq.\eqref{ijnbhy} is a vector space.\\
Now let us consider the following fractal matrix:
\begin{equation}\label{reeeeeq}
  \mathbf{X}(t)=\left(
                  \begin{array}{ccc}
                    x_{11}(t) & \cdots & x_{1n}(t) \\
                    \vdots &  & \vdots \\
                    x_{n1}(t) & \cdots & x_{nn}(t) \\
                  \end{array}
                \right)
\end{equation}
whose columns are the fractal vectors $\mathbf{x}^{(1)}(t),...,\mathbf{x}^{(n)}(t)$. The Wronskian of the solutions is defined by:
\begin{equation}\label{ijnbgt}
  W[\mathbf{x}^{(1)}(t),...,\mathbf{x}^{(n)}(t)]=\det\mathbf{X}(t),~~\forall t\in F.
\end{equation}

\begin{theorem}
If  the Wronskian of the solutions is not zero in each point of  $F$, then there exist constants $c_{1},c_{2},...,c_{n}$ such that each solution $\mathbf{x}=\mathbf{x}(t)$ of the system Eq.\eqref{ijnbhy} can be expressed as:
  \begin{equation}\label{iojku}
    \mathbf{x}(t)=c_{1}\mathbf{x}^{(1)}(t)+...+c_{n}\mathbf{x}^{(n)}(t)
  \end{equation}
  This is called the general solution of the system Eq.\eqref{ijnbhy}, and the set ${\mathbf{x}^{(1)}(t), \ldots, \mathbf{x}^{(n)}(t)}$ is denoted as the fundamental set of solutions on  $F$.
\end{theorem}
\begin{proof}
  To prove this theorem, let us consider a point $t_{0} \in F$, and let $\mathbf{y}(t)=\mathbf{x}(t)$. We have to show that there exist values $c_{1}, \ldots, c_{n}$ such that:
\begin{equation}\label{iouj}
c_{1}\mathbf{x}^{(1)}(t_{0})+ \ldots +c_{n}\mathbf{x}^{(n)}(t_{0})=\mathbf{y}(t_{0})
\end{equation}
This equation can be expressed component-wise as:
\begin{equation}\label{uiit}
  \begin{split}
    c_{1}x_{11}(t_{0})+...+c_{n}x_{1n}(t_{0})&=y_{1}(t_{0})\\
    &\vdots\\
    c_{1}x_{n1}(t_{0})+...+c_{n}x_{nn}(t_{0})&=y_{n}(t_{0})
    \end{split}
  \end{equation}
By the conjugacy of $F^{\alpha}$-calculus and the ordinary calculus \cite{parvate2009calculus,parvate2011calculus,Alireza-book}), the necessary and sufficient condition for Eq.\eqref{uiit}  to have a unique solution is the non-vanishing of the determinant of the coefficients, which is the Wronskian evaluated at $t_{0}$. Therefore, there exists a unique solution of Equation \eqref{ijnbhy} in the form $\mathbf{x}(t)=c_{1}\mathbf{x}^{(1)}(t)+ \ldots +c_{n}\mathbf{x}^{(n)}(t)$. This completes the proof.
\end{proof}

\begin{theorem}
If $\mathbf{x}^{(1)}(t),\ldots,\mathbf{x}^{(n)}(t)$ are solutions of Eq.\eqref{ijnbhy} on the fractal set $F$, then at each point of $F \subset \mathbb{R}$, $W[ \mathbf{x}^{(1)}(t),\ldots,\mathbf{x}^{(n)}(t)]$ is either identically zero or never vanishes.
\end{theorem}
\begin{proof}
To prove this theorem, we establish that the Wronskian of $\mathbf{x}^{(1)}(t),\ldots,\mathbf{x}^{(n)}(t)$ satisfies the differential equation:
\begin{equation}\label{ijnhy}
D_{F}^{\alpha}W=(p_{11}(t)+p_{22}(t)+\ldots+p_{nn}(t))W
\end{equation}
Thus, we have:
\begin{equation}\label{oool2}
W(t)=c\exp\left(\int(p_{11}(t)+\ldots+p_{nn}(t))d_{F}^{\alpha}t\right)
\end{equation}
where $c$ is an arbitrary constant. The conclusion of the theorem follow immediately.
\end{proof}

\begin{theorem}
Let
\begin{equation}\label{olpppk}
    \mathbf{e}^{(1)}=\left(
              \begin{array}{c}
                1 \\
                0 \\
                0 \\
                \vdots \\
                0 \\
              \end{array}
            \right)
 ~~~~~\mathbf{e}^{(2)}=\left(
              \begin{array}{c}
                0 \\
                1 \\
                0 \\
                \vdots \\
                0 \\
              \end{array}
            \right),\cdots,\mathbf{e}^{(n)}=\left(
              \begin{array}{c}
                0 \\
                0 \\
                0 \\
                \vdots \\
                1 \\
              \end{array}
            \right).
  \end{equation}
Further, let $\mathbf{x}^{(1)}(t),\ldots,\mathbf{x}^{(n)}(t)$ be the solutions of the following system of fractal differential equation
\begin{equation}\label{tt847wqqqq5}
  D_{F}^{\alpha}\mathbf{x}(t)=\mathbf{P}(t)\mathbf{x}(t),~~~\forall t\in F,
\end{equation}
that satisfy the initial conditions:
\begin{equation}\label{iokjnhy}
\mathbf{x}^{(1)}(t_{0})=\mathbf{e}^{(1)},\ldots,\mathbf{x}^{(n)}
(t_{0})=\mathbf{e}^{(n)},
\end{equation}
where $t_{0}\in F$. Then $\mathbf{x}^{(1)}(t),\ldots,\mathbf{x}^{(n)}(t)$ form a fundamental set of solutions of the system Eq.\eqref{tt847wqqqq5}.
\end{theorem}
\begin{proof}
\textbf{Linear Independence:}\\
By contradiction let us suppose that there exist constants $c_1, c_2, \ldots, c_n$ not all zero such that
\begin{equation}\label{iokju9}
  c_1 \mathbf{x}^{(1)}(t) + c_2 \mathbf{x}^{(2)}(t) + \cdots + c_n \mathbf{x}^{(n)}(t) = \mathbf{0},
\end{equation}
for all $t \in F$. Now by evaluating this expression at $t = t_0$, we have that
\begin{equation}\label{iokmnj58}
  c_1 \mathbf{x}^{(1)}(t_0) + c_2 \mathbf{x}^{(2)}(t_0) + \cdots + c_n \mathbf{x}^{(n)}(t_0) = \mathbf{0},
\end{equation}
and by the initial conditions \eqref{iokjnhy}, we get
\begin{equation}\label{ollmmj5}
  c_1 \mathbf{e}^{(1)} + c_2 \mathbf{e}^{(2)} + \cdots + c_n \mathbf{e}^{(n)} = \mathbf{0},
\end{equation}
which implies that $c_1 = c_2 = \cdots = c_n = 0$. \\
\textbf{Spanning the Solution Space:}\\
To show that $\mathbf{x}^{(1)}(t), \ldots, \mathbf{x}^{(n)}(t)$ span the solution space, let us denoted by  $\mathbf{x}(t)$  any solution of the system Eq.\eqref{tt847wqqqq5} with initial conditions \eqref{iokjnhy}, and let us define $\mathbf{y}^{(i)}(t) = \mathbf{x}(t) - \mathbf{x}^{(i)}(t)$ for $i = 1, 2, \ldots, n$. Then $\mathbf{y}^{(i)}(t)$ satisfies the same differential equation as $\mathbf{x}(t)$ and has the initial condition $\mathbf{y}^{(i)}(t_0) = \mathbf{0}$. By the uniqueness of solutions to initial value problems (see \cite{}), we have $\mathbf{y}^{(i)}(t) = \mathbf{0}$ for all $i = 1, 2, \ldots, n$. This implies that $\mathbf{x}(t) = \mathbf{x}^{(1)}(t) + \mathbf{x}^{(2)}(t) + \cdots + \mathbf{x}^{(n)}(t)$. Hence, any solution $\mathbf{x}(t)$ of Eq.\eqref{tt847wqqqq5} with initial conditions \eqref{iokjnhy} can be expressed as a linear combination of $\mathbf{x}^{(1)}(t), \ldots, \mathbf{x}^{(n)}(t)$. Thus, $\mathbf{x}^{(1)}(t), \ldots, \mathbf{x}^{(n)}(t)$ span the solution space.
Since $\mathbf{x}^{(1)}(t), \ldots, \mathbf{x}^{(n)}(t)$ are both linearly independent and span the solution space, they form a fundamental set of solutions of the system Eq.\eqref{tt847wqqqq5}. Thus, the theorem is proved.
\end{proof}

\begin{theorem}
Consider the system of fractal differential equation as:
\begin{equation}\label{oollk}
D_{F}^{\alpha}\mathbf{x}(t)=\mathbf{P}(t)\mathbf{x}(t),
\end{equation}
where each element of $\mathbf{P}$ is a real-valued $F$-continuous function on $F$. If $\mathbf{x}(t)=\mathbf{u}(t)+i\mathbf{v}(t)$ is a complex-valued solution of Eq.\eqref{oollk}, then its real part $\mathbf{u}(t)$ and its imaginary part $\mathbf{v}(t)$ are also solutions of this equation.
\end{theorem}
\begin{proof}
To prove this theorem, we substitute $\mathbf{u}(t)+ i\mathbf{v}(t)$ for $\mathbf{x}(t)$ in Eq.\eqref{oollk}, thereby obtaining
\begin{equation}\label{iombx}
D_{F}^{\alpha}\mathbf{x}(t)-\mathbf{P}(t)\mathbf{x}(t)=
D_{F}^{\alpha}\mathbf{u}(t)-\mathbf{P}(t)\mathbf{u}(t)
+i( D_{F}^{\alpha}\mathbf{v}(t)-\mathbf{P}(t)\mathbf{v}(t))=0
\end{equation}
Thus we have
\begin{equation}\label{iooo47}
D_{F}^{\alpha}\mathbf{u}(t)-\mathbf{P}(t)\mathbf{u}(t)=0,~~~\text{and}~~~
D_{F}^{\alpha}\mathbf{v}(t)-\mathbf{P}(t)\mathbf{v}(t)=0,
\end{equation}
Therefore, $\mathbf{u}(t)$ and $\mathbf{v}(t)$ are solutions of Eq.\eqref{oollk}.
\end{proof}

\section{Homogeneous System of $\alpha$-Order  Linear Fractal Differential Equation with Constant Coefficients \label{3g}}
A homogeneous fractal linear system with constant coefficients is a system of linear differential equations in which each element of the fractal matrix is constant. Mathematically, such a system can be represented as:
\begin{equation}\label{33oooplmj}
D_{F}^{\alpha}\mathbf{x}(t)=\mathbf{A} \mathbf{x}(t),
\end{equation}
where $\mathbf{A}$ is a $n\times n$ constant  matrix. By  \cite{golmankhaneh2023solving} it follows that to solve Eq.\eqref{33oooplmj}, we consider a solution of the form
\begin{equation}\label{eeeq12}
\mathbf{x}(t)=\boldsymbol{\xi}\exp(rS_{F}^{\alpha}(t)),
\end{equation}
where $r$ and $\boldsymbol{\xi}$ are to be determined. Substituting Eq.\eqref{eeeq12} into Eq.\eqref{33oooplmj}, we obtain
\begin{equation}\label{rwqaqz891}
r \boldsymbol{\xi}\exp(rS_{F}^{\alpha}(t))=
\mathbf{A}\boldsymbol{\xi}\exp(rS_{F}^{\alpha}(t)),
\end{equation}
or equivalently,
\begin{equation}\label{io18}
(\mathbf{A}-r\mathbf{I})\boldsymbol{\xi}=0,
\end{equation}
where $\mathbf{I}$ is the $n\times n$ identity matrix therefore, to find solutions of Eq.\eqref{33oooplmj}, we solve the system of algebraic equations given by Eq.\eqref{io18}. The eigenvalues $r_{1},...,r_{n}$ are the roots of the $n^{th}$ degree polynomial equation
\begin{equation}\label{dddddq}
\det(\mathbf{A}-r\mathbf{I})=0.
\end{equation}
The corresponding eigenvectors to eigenvalues $r_{1},...,r_{n}$ provide the general solution of the fractal system given by Eq.\eqref{33oooplmj}.
\begin{example}
Consider the fractal system of differential equations given by:
\begin{equation}\label{frreq1}
D_{F}^{\alpha}\mathbf{x}(t)=\left(
\begin{array}{cc}
1 & 1 \\
4 & 1 \\
\end{array}
\right)\mathbf{x}(t).
\end{equation}
To find the solution explicitly, we assume that $\mathbf{x}(t)= \boldsymbol{\xi}\exp(rS_{F}^{\alpha}(t))$ and substitute it into Eq.\eqref{frreq1}, we have:
\begin{equation}\label{ree8569}
\left(
\begin{array}{cc}
1-r & 1 \\
4 & 1-r \\
\end{array}
\right)\left(
\begin{array}{c}
\xi_{1} \\
\xi_{2} \\
\end{array}
\right)=\left(
\begin{array}{c}
0 \\
0 \\
\end{array}
\right).
\end{equation}
This equation leads to the characteristic equation:
\begin{equation}\label{io4875}
\begin{vmatrix}
1-r & 1 \\
4 & 1-r
\end{vmatrix}=(r-3)(r+1)=0,
\end{equation}
which gives the eigenvalues $r_{1}=3$ and $r_{2}=-1$. These eigenvalues are also the roots of the characteristic polynomial. Substituting $r=3$ and $r=-1$ into Eq.\eqref{ree8569}, we obtain the eigenvectors:
\begin{equation}\label{ewaqsdcxd}
\boldsymbol{\xi}^{(1)}(t)=\left(
\begin{array}{c}
1 \\
2 \\
\end{array}
\right),~~~~~\boldsymbol{\xi}^{(2)}=\left(
\begin{array}{c}
1 \\
-2 \\
\end{array}
\right).
\end{equation}

Thus, the corresponding solutions of Eq.\eqref{frreq1} are:
\begin{equation}\label{ttyy744}
\mathbf{x}^{(1)}(t)=\left(
\begin{array}{c}
1 \\
2 \\
\end{array}
\right)\exp(3S_{F}^{\alpha}(t)),~~~~\mathbf{x}^{(1)}=\left(
\begin{array}{c}
1 \\
-2 \\
\end{array}
\right)\exp(-S_{F}^{\alpha}(t)).
\end{equation}

The Wronskian of these solutions is:
\begin{equation}\label{ewws123}
W[\mathbf{x}^{(1)}(t),\mathbf{x}^{(2)(t)}]=\begin{vmatrix}
\exp(3S_{F}^{\alpha}(t)) & \exp(-S_{F}^{\alpha}(t)) \\
2\exp(3S_{F}^{\alpha}(t)) & -2\exp(-S_{F}^{\alpha}(t))
\end{vmatrix}=-4\exp(2S_{F}^{\alpha}(t)).
\end{equation}
This Wronskian is not zero, indicating that the solutions $\mathbf{x}^{(1)}(t)$ and $\mathbf{x}^{(2)}(t)$ form a fundamental set.

Therefore, the general solution of the system given by Eq.\eqref{frreq1} is:
\begin{equation}\label{erewwsdfr}
\begin{split}
\mathbf{x}(t)&=c_{1}\mathbf{x}^{1}(t)+c_{2}\mathbf{x}^{2}(t)\\&=
c_{1}\left(
\begin{array}{c}
1 \\
2 \\
\end{array}
\right)\exp(3S_{F}^{\alpha}(t))+c_{2}\left(
\begin{array}{c}
1 \\
-2 \\
\end{array}
\right)\exp(-S_{F}^{\alpha}(t))
\end{split}
\end{equation}
where $c_{1}$ and $c_{2}$ are arbitrary constants. Setting $c_{2}=0$ results in:
\begin{equation}\label{ijnjk}
\mathbf{x}=c_{1}\mathbf{x}^{1}(t),
\end{equation}
or in scalar form:
\begin{equation}\label{iokmnj}
x_{1}(t)=c_{1}\exp(3S_{F}^{\alpha}(t)),~~~~x_{2}(t)=2c_{1}\exp(3S_{F}^{\alpha}(t)).
\end{equation}
Choosing $c_{1}=0$, yields:
\begin{equation}\label{ookijnjk}
\mathbf{x}(t)=c_{2}\mathbf{x}^{2}(t),
\end{equation}
or:
\begin{equation}\label{iokmnj66}
x_{1}(t)=c_{2}\exp(-S_{F}^{\alpha}(t)),~~~~x_{2}(t)=-2c_{2}\exp(-S_{F}^{\alpha}(t)).
\end{equation}
\begin{figure}[ht]
  \centering
  \includegraphics[scale=0.5]{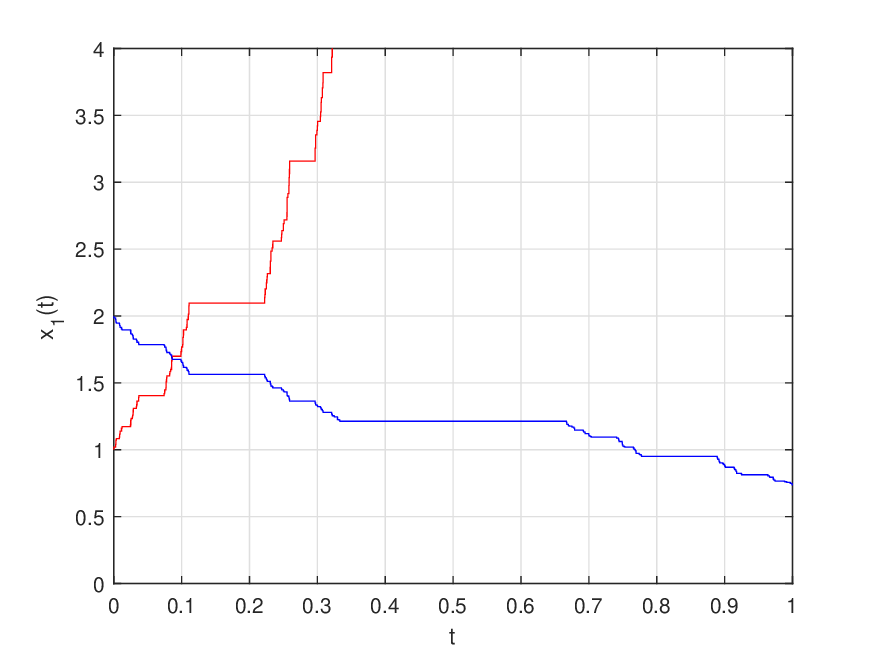}
  \caption{We have plotted $x_{1}$ for $c_{1}=1$ (red) and $c_{2}=2$ (blue). }\label{ddsss-1}
\end{figure}
\begin{figure}[ht]
  \centering
  \includegraphics[scale=0.5]{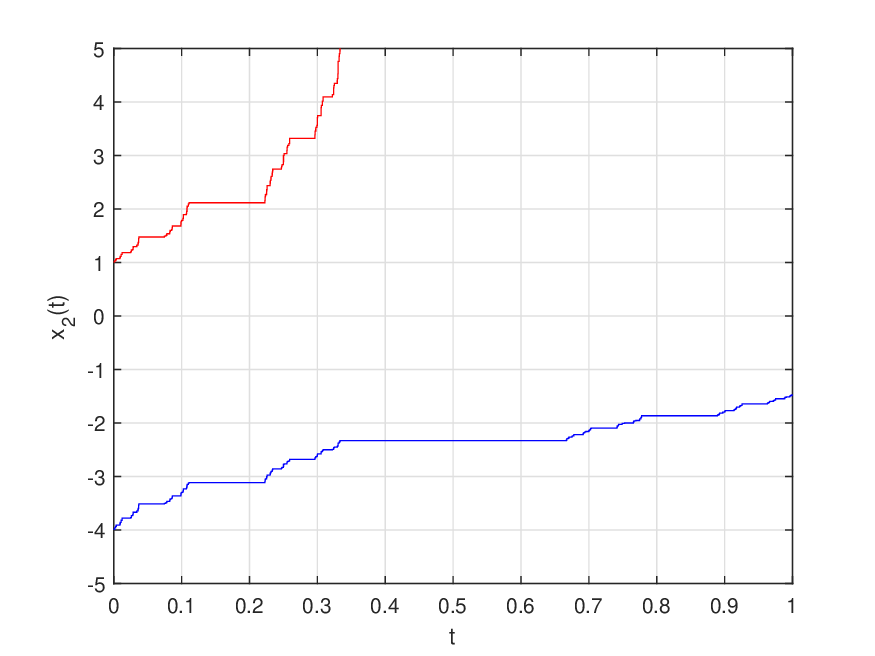}
  \caption{We have plotted $x_{2}(t)$ for $c_{1}=1/2$ (red) and $c_{2}=2$ (blue). }\label{ddsss-2}
\end{figure}
In Figures \ref{ddsss-1} and \ref{ddsss-2}, we have visualized Eqs.\eqref{iokmnj} and \eqref{iokmnj66} for various value of the constants $c_{1}$ and $c_{2}$.
\end{example}

\begin{example}
  Consider the system of fractal differential equations given by:
\begin{equation}\label{rteew}
D_{F}^{\alpha}\mathbf{x}(t)=\left(
\begin{array}{cc}
-3 & \sqrt{2} \\
\sqrt{2} & -2 \\
\end{array}
\right)\mathbf{x}(t).
\end{equation}
To find the solution, we assume $\mathbf{x}(t)=\boldsymbol{\xi}\exp(rS_{F}^{\alpha}(t))$. Substituting this into Eq.\eqref{rteew}, we obtain the algebraic system:
\begin{equation}\label{ree}
\left(
\begin{array}{cc}
-3-r & \sqrt{2} \\
\sqrt{2} & -2-r \\
\end{array}
\right)\left(
\begin{array}{c}
\xi_{1} \\
\xi_{2} \\
\end{array}
\right)=\left(
\begin{array}{c}
0 \\
0 \\
\end{array}
\right).
\end{equation}

So, the eigenvalues and eigenvectors are:
\begin{equation}\label{iiiip}
r_{1}=-1,~~~~\boldsymbol{\xi}^{1}=\left(
\begin{array}{c}
1 \\
\sqrt{2} \\
\end{array}
\right),~~~~r_{2}=-4,~~~~\boldsymbol{\xi}^{2}=\left(
\begin{array}{c}
-\sqrt{2} \\
1 \\
\end{array}
\right).
\end{equation}

Thus, the general solution is:
\begin{equation}\label{iijjj852}
\begin{split}
\mathbf{x}(t)&=c_{1}\mathbf{x}^{1}(t)+c_{2}\mathbf{x}^{2}(t)\\&=
c_{1}\left(
\begin{array}{c}
1 \\
\sqrt{2}\\
\end{array}
\right)\exp(-S_{F}^{\alpha}(t))+c_{2}\left(
\begin{array}{c}
-\sqrt{2} \\
1\\
\end{array}
\right)\exp(-4S_{F}^{\alpha}(t)).
\end{split}
\end{equation}
or if $c_{2}=0 $
\begin{equation}\label{re12eeawq}
  x_{1}(t)=c_{1}\exp(-S_{F}^{\alpha}(t)),~~~~~~~~~
  x_{2}(t)=\sqrt{2}~c_{1}\exp(-S_{F}^{\alpha}(t))
  \end{equation}
    and if $c_{1}=0$
    \begin{equation}\label{32re12eeawq}
 x_{1}(t)=-\sqrt{2}~c_{2}\exp(-4S_{F}^{\alpha}(t)),~~~~~
  x_{2}(t)=c_{2}\exp(-4S_{F}^{\alpha}(t))
\end{equation}
\begin{figure}[ht]
  \centering
  \includegraphics[scale=0.5]{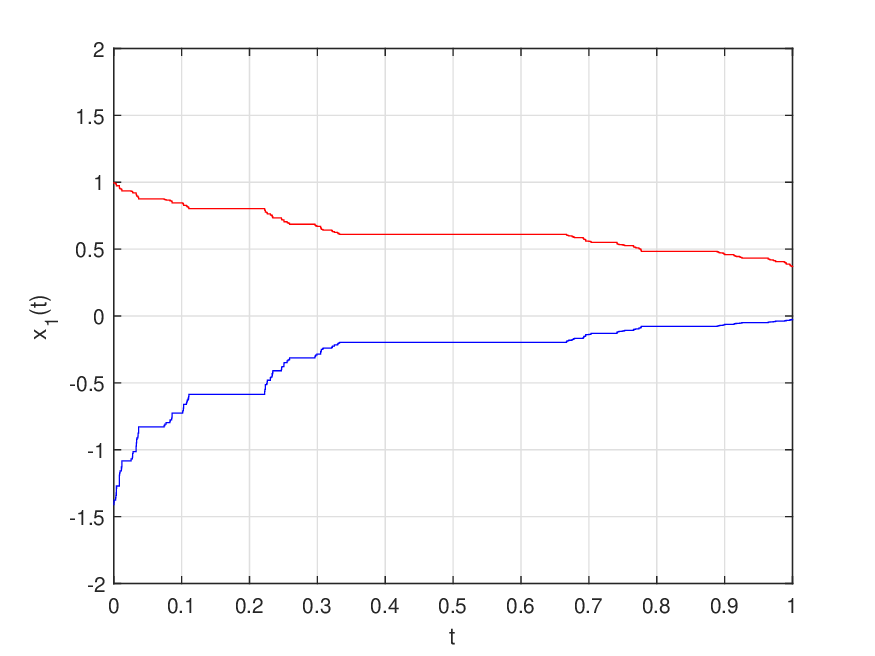}
  \caption{Graph of $x_{1}(t)$ for $c_{1}=1$ and $c_{2}=1$}\label{re588}
\end{figure}

\begin{figure}[ht]
  \centering
 \includegraphics[scale=0.5]{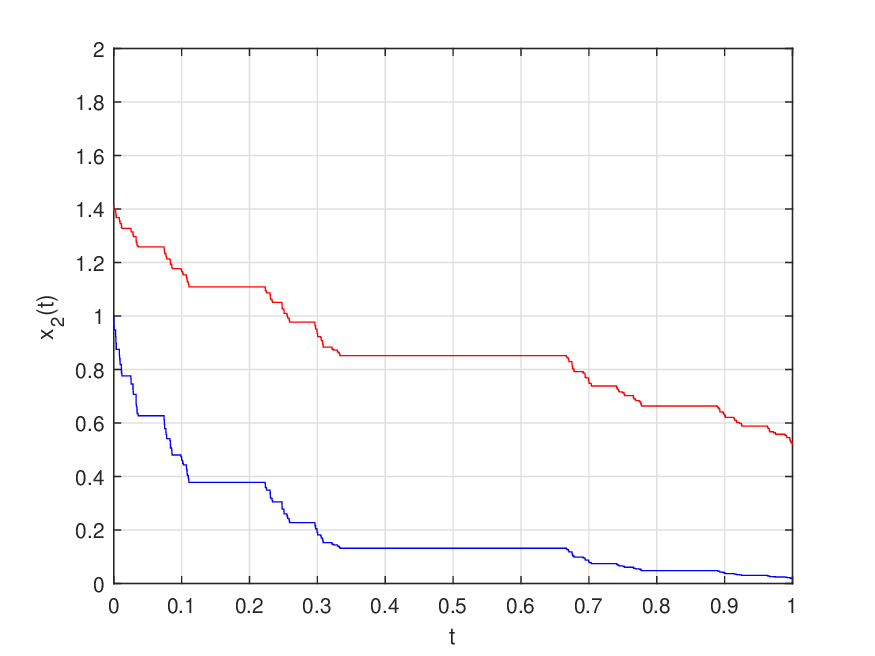}
  \caption{Graph of $x_{2}(t)$ for $c_{1}=1$ and $c_{2}=1$}\label{588}
\end{figure}
In Figures \ref{re588} and \ref{588}, we have plotted  Eq.\eqref{re12eeawq} and Eq.\eqref{32re12eeawq} for different values of $c_{1}$ and $c_{2}$.

\end{example}

\begin{example}
Consider the system of fractal differential equations given by:
\begin{equation}\label{i9uyhbgt}
D_{F}^{\alpha}\mathbf{x}(t)=\left(
\begin{array}{cc}
-\frac{1}{2} & 1 \\
-1 & -\frac{1}{2} \\
\end{array}
\right)\mathbf{x}.
\end{equation}

To find the solution, we assume $\mathbf{x}(t)=\boldsymbol{\xi} \exp(r S_{F}^{\alpha}(t))$. Substituting this into Eq.\eqref{i9uyhbgt} we have:
\begin{equation}\label{iouhyjpkmu}
\left(
\begin{array}{cc}
-\frac{1}{2}-r & 1 \\
-1 & -\frac{1}{2}-r \\
\end{array}
\right)
\left(
\begin{array}{c}
\xi_{1} \\
\xi_{2} \\
\end{array}
\right)=
\left(
\begin{array}{c}
0 \\
0\\
\end{array}
\right).
\end{equation}

Then, the eigenvalues and eigenvectors are:
\begin{equation}\label{ewwsw1}
r_{1}=-\frac{1}{2}+i,~~~\boldsymbol{\xi}^{1}=\left(
\begin{array}{c}
1 \\
i \\
\end{array}
\right),~~~
r_{2}=-\frac{1}{2}-i,~~~\boldsymbol{\xi}^{2}=\left(
\begin{array}{c}
1 \\
-i \\
\end{array}
\right),
\end{equation}

and the general solution is:
\begin{equation}\label{yy45}
\begin{split}
\mathbf{x}(t)&=c_{1}\mathbf{x}^{1}(t)+c_{2}\mathbf{x}^{2}(t)\\&=
c_{1}\left(\begin{array}{c}
1 \\
i \\
\end{array}
\right)\exp((-\frac{1}{2}+i)S_{F}^{\alpha}(t))+
c_{2}\left(\begin{array}{c}
1 \\
-i \\
\end{array}
\right)\exp((-\frac{1}{2}-i)S_{F}^{\alpha}(t)).
\end{split}
\end{equation}

We can express the real-valued solution as:
\begin{equation}\label{yttt23}
\begin{split}
\mathbf{u}(t)&=\exp\left(-\frac{S_{F}^{\alpha}(t)}{2}\right)\left(
\begin{array}{c}
\cos(S_{F}^{\alpha}(t)) \\
-\sin(S_{F}^{\alpha}(t)) \\
\end{array}
\right),\\
\mathbf{v}(t)&=\exp\left(-\frac{S_{F}^{\alpha}(t)}{2}\right)\left(
\begin{array}{c}
\sin(S_{F}^{\alpha}(t)) \\
\cos(S_{F}^{\alpha}(t)) \\
\end{array}
\right).
\end{split}
\end{equation}
Then the general solution is
\begin{equation}\label{ees23}
  \mathbf{x}(t)=c_{1}\mathbf{u}(t)+c_{2}\mathbf{v}(t)
\end{equation}
or if $c_{1}=0$
\begin{equation}\label{t544aqwrr}
  x_{1}(t)=c_{2}\exp\left(-\frac{S_{F}^{\alpha}(t)}{2}\right)\sin(S_{F}^{\alpha}(t)),~~~
  x_{2}(t)=c_{2}\exp\left(-\frac{S_{F}^{\alpha}(t)}{2}\right)\cos(S_{F}^{\alpha}(t))
\end{equation}
or if $c_{2}=0$
\begin{equation}\label{twsaq544aqwrr}
  x_{1}(t)=c_{1}\exp\left(-\frac{S_{F}^{\alpha}(t)}{2}\right)\cos(S_{F}^{\alpha}(t)),~~~
  x_{2}(t)=-c_{1}\exp\left(-\frac{S_{F}^{\alpha}(t)}{2}\right)\sin(S_{F}^{\alpha}(t))
\end{equation}
Their Wronskian is:
\begin{equation}\label{er}
 \begin{split}
   W[\mathbf{u},\mathbf{v}](t)&=\begin{vmatrix}
  \exp\left(-\frac{S_{F}^{\alpha}(t)}{2}\right)\cos(S_{F}^{\alpha}(t)) & \exp\left(-\frac{S_{F}^{\alpha}(t)}{2}\right)\sin(S_{F}^{\alpha}(t))\\
   -\exp\left(-\frac{S_{F}^{\alpha}(t)}{2}\right)\sin(S_{F}^{\alpha}(t)) & \exp\left(-\frac{S_{F}^{\alpha}(t)}{2}\right)\cos(S_{F}^{\alpha}(t))\\
   \end{vmatrix}\\&=\exp\left(-S_{F}^{\alpha}(t)\right)
   \end{split}
 \end{equation}
which is never zero. Thus, $\mathbf{u}(t)$ and $\mathbf{v}(t)$ are fundamental solutions of Eq.\eqref{i9uyhbgt}.

\begin{figure}[ht]
  \centering
  \includegraphics[scale=0.5]{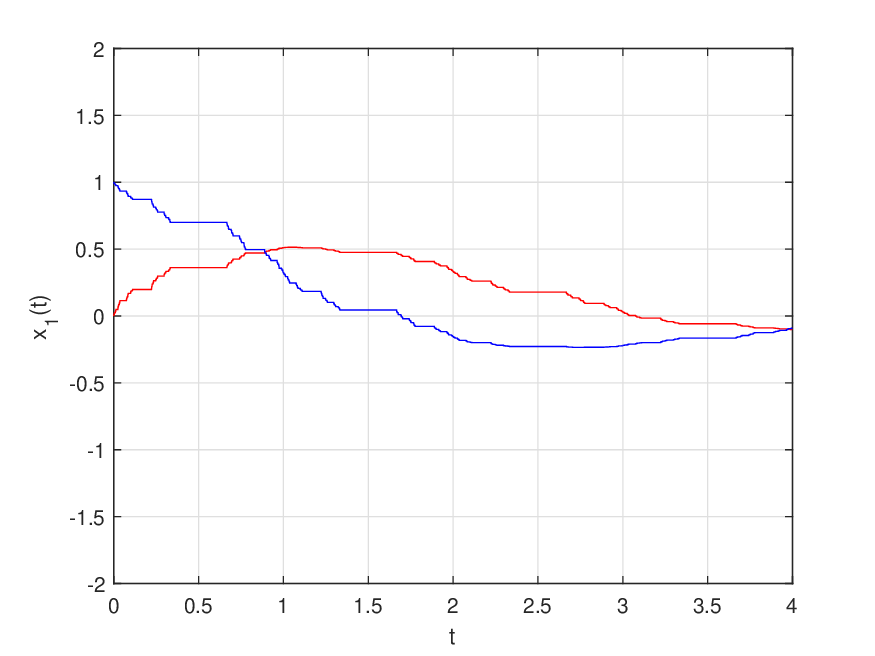}
  \caption{Graph of $x_{1}$ for different values of $c_{1}$ and $c_{2}$.}\label{ewrtfdcd}
\end{figure}

\begin{figure}[ht]
  \centering
  \includegraphics[scale=0.5]{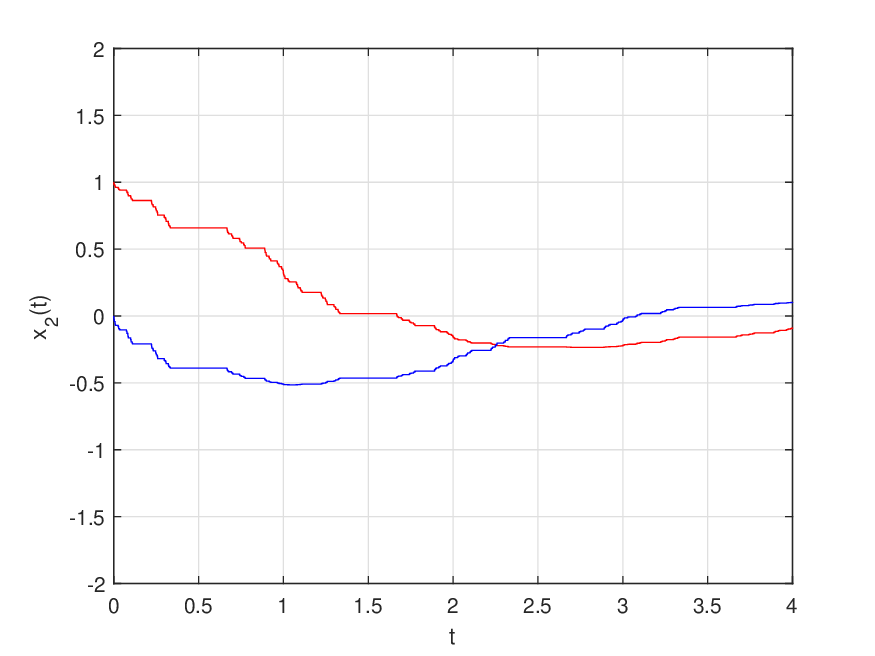}
  \caption{Graph of $x_{2}$ for different values of $c_{1}$ and $c_{2}$.}\label{ewrtfdqqcd}
\end{figure}

Graphs depicting the behavior of $x_{1}$ and $x_{2}$ in Eqs.\eqref{t544aqwrr} and \eqref{twsaq544aqwrr} with respect to $t$ are illustrated in Figures \ref{ewrtfdcd} and \ref{ewrtfdqqcd}.
\end{example}

\begin{example}
Consider the system of fractal differential equations given by:
\begin{equation}\label{eee122qq3aw}
D_{F}^{\alpha}\mathbf{x}(t)=\left(
\begin{array}{cc}
1 & -1 \\
1 & 3 \\
\end{array}
\right)\mathbf{x}(t).
\end{equation}
To solve Eq.\eqref{eee122qq3aw}, let $\mathbf{x}(t)=\boldsymbol{\xi}\exp(rS_{F}^{\alpha}(t))$. Substituting this into Eq.\eqref{eee122qq3aw}, we find the double eigenvalue:
\begin{equation}\label{reeeqaw}
r_{1}=r_{2}=2,~~~~\boldsymbol{\xi}=\left(\begin{array}{c}
1 \\
-1 \\
\end{array}
\right)
\end{equation}
One independent solution is then:
\begin{equation}\label{tee}
\mathbf{x}^{(1)}(t)=\left(\begin{array}{c}
1 \\
-1 \\
\end{array}
\right)\exp(2S_{F}^{\alpha}(t))
\end{equation}
To find the second independent solution, we assume:
\begin{equation}\label{reeeqsaw}
\mathbf{x}=\boldsymbol{\xi} S_{F}^{\alpha}(t)\exp(2S_{F}^{\alpha}(t))+
\boldsymbol{\eta}\exp(2S_{F}^{\alpha}(t))
\end{equation}
where $\boldsymbol{\xi}$ and $\boldsymbol{\eta}$ are constant vectors to be determined. Substituting Eq.\eqref{reeeqsaw} into Eq.\eqref{eee122qq3aw}, we get:
\begin{equation}\label{iokmju}
(\mathbf{A}-2\mathbf{I})\boldsymbol{\xi}=0,~~~~
(\mathbf{A}-2\mathbf{I})\boldsymbol{\eta}=\boldsymbol{\xi}
\end{equation}
By solving Eq.\eqref{iokmju}, we obtain:
\begin{equation}\label{eerq23}
r_{1}=r_{2}=2,~~~~\boldsymbol{\xi}=\left(
\begin{array}{c}
1 \\
-1 \\
\end{array}
\right)
\end{equation}
and
\begin{equation}\label{iookmju85}
\boldsymbol{\eta}=\left(
\begin{array}{c}
0 \\
-1 \\
\end{array}
\right)+k \left(
\begin{array}{c}
1 \\
-1 \\
\end{array}
\right)
\end{equation}
where $k$ is an arbitrary constant. Substituting $\boldsymbol{\xi}$ and $\boldsymbol{\eta}$ into Eq.\eqref{reeeqsaw}, we find the second independent solution:
\begin{equation}\label{ioomkjnu}
\mathbf{x}^{(2)}(t)=\left(
\begin{array}{c}
1 \\
-1 \\
\end{array}
\right)S_{F}^{\alpha}(t)\exp(2S_{F}^{\alpha}(t))+
\left(
\begin{array}{c}
0 \\
-1 \\
\end{array}
\right)\exp(2S_{F}^{\alpha}(t))
\end{equation}
The Wronskian of these solutions is:
\begin{equation}\label{reeew}
W[\mathbf{x}^{(1)}(t),\mathbf{x}^{(2)}(t)]=-\exp(4S_{F}^{\alpha}(t))
\end{equation}
Thus, $\mathbf{x}^{(1)}(t)$ and $\mathbf{x}^{(2)}(t)$ form a fundamental set of solutions. The general solution of Eq.\eqref{eee122qq3aw} is then given by:
\begin{equation}\label{ijnhyuhb}
\begin{split}
\mathbf{x}(t)&=c_{1}\mathbf{x}^{(1)}(t)+c_{2}\mathbf{x}^{(2)}(t)\\&=
c_{1}\left(
\begin{array}{c}
1 \\
-1 \\
\end{array}
\right)\exp(2S_{F}^{\alpha}(t))+c_{2}\bigg[\left(
\begin{array}{c}
1 \\
-1 \\
\end{array}
\right)
S_{F}^{\alpha}(t)\exp(2S_{F}^{\alpha}(t))\\&+
\left(
\begin{array}{c}
0 \\
-1 \\
\end{array}
\right)\exp(2S_{F}^{\alpha}(t))\bigg]
\end{split}
\end{equation}
or if $c_{1}=0$ we have
\begin{equation}\label{77747m}
  x_{1}(t)=c_{2}S_{F}^{\alpha}(t)\exp(2S_{F}^{\alpha}(t)),~~~
  x_{2}(t)=-c_{2}(S_{F}^{\alpha}(t)+1)\exp(2S_{F}^{\alpha}(t))
\end{equation}
if $c_{2}=0$ we have
\begin{equation}\label{7yhbgt7747m}
  x_{1}(t)=c_{1}\exp(2S_{F}^{\alpha}(t)),~~~
  x_{2}(t)=-c_{1}\exp(2S_{F}^{\alpha}(t))
\end{equation}

\begin{figure}[ht]
  \centering
  \includegraphics[scale=0.5]{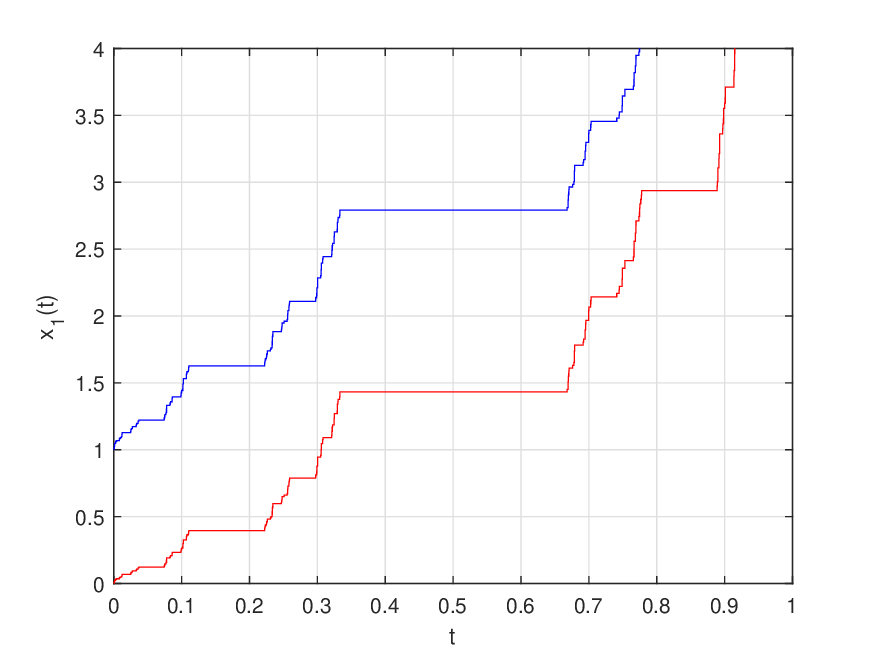}
  \caption{Graph of $x_{1}$ for different values of $c_{1}=1$ and $c_{2}=1$.}\label{rrftre}
\end{figure}

\begin{figure}[ht]
  \centering
  \includegraphics[scale=0.5]{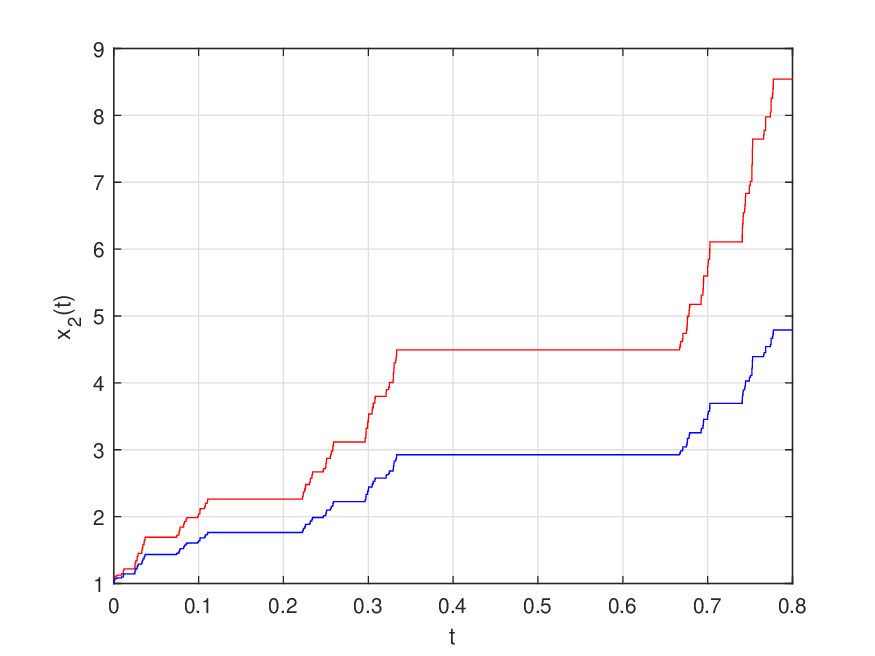}
  \caption{Graph of $x_{2}$ for different values of $c_{1}=-1$ and $c_{2}=-1$.}\label{rrfcxrftre}
\end{figure}

Graphs of $x_{1}$ and $x_{2}$ in Eqs.\eqref{77747m} and \eqref{7yhbgt7747m} are shown in Figures \ref{rrftre} and \ref{rrfcxrftre}.
\end{example}

\begin{example}
Consider the system of fractal differential equations represented as:
\begin{equation}\label{r4ttts}
D_{F}^{\alpha}\mathbf{x}(t)=\left(
\begin{array}{cc}
2 & 0 \\
0 & -3 \\
\end{array}
\right)
\mathbf{x}(t).
\end{equation}

Expressing Eq.\eqref{r4ttts} in the scalar form, we obtain that:
\begin{equation}\label{ewww1221}
D_{F}^{\alpha}x_{1}(t)=2x_{1}(t), \quad D_{F}^{\alpha}x_{2}(t)=2x_{2}(t).
\end{equation}

The solution to Eq.\eqref{ewww1221} is:
\begin{equation}\label{erwwsw}
\mathbf{x}(t)=\left(
\begin{array}{c}
c_{1}\exp(2S_{F}^{\alpha}(t)) \\
c_{2}\exp(-3S_{F}^{\alpha}(t)) \\
\end{array}
\right)=c_{1}\left(
\begin{array}{c}
\exp(2S_{F}^{\alpha}(t)) \\
0 \\
\end{array}
\right)+c_{2}\left(
\begin{array}{c}
\exp(-3S_{F}^{\alpha}(t)) \\
0 \\
\end{array}
\right).
\end{equation}

Alternatively, we can express it as:
\begin{equation}\label{yy7777}
\mathbf{x}^{1}(t)=\left(
\begin{array}{c}
1 \\
0 \\
\end{array}
\right)\exp(2S_{F}^{\alpha}(t)), \quad \mathbf{x}^{2}(t)=\left(
\begin{array}{c}
0 \\
1 \\
\end{array}
\right)\exp(-3S_{F}^{\alpha}(t)).
\end{equation}

The Wronskian of these two solutions is:
\begin{equation}\label{tt55433}
W[\mathbf{x}^{1}(t),\mathbf{x}^{2}(t)]=\begin{vmatrix}
\exp(2S_{F}^{\alpha}(t)) & 0 \\
0 & \exp(-3S_{F}^{\alpha}(t)) \\
\end{vmatrix}=\exp(-S_{F}^{\alpha}(t)),
\end{equation}
which is never zero. Hence, $\mathbf{x}^{1}(t)$ and $\mathbf{x}^{2}(t)$ constitute a fundamental set of solutions.
\end{example}

\section{Conclusion \label{4g}}
In this paper we have introduced the notion of a fractal matrix and its corresponding homogeneous system of fractal differential equations,
 through the use of $F^{\alpha}$-calculus,  introduced by Parvate and Gangal  \cite{parvate2009calculus} and subsequently studied by  Khalili Golmankhaneh  \cite{Alireza-book}. As it is well known the systems of linear differential equations are of great importance in literature, both because many physical phenomena can be modeled by equations of this type and because even non-linear phenomena can, in a first approximation, be described by linear equations. Therefore, motivated by this fact, we dealt with the detailed study of the set of solutions of homogeneous system of $\alpha$-order linear differential equations: $D^{\alpha}_F\bold{x}(t)=\bold{P}(t)\bold{x}(t),$  having non-integer orders of differentiation in the system. Moreover,  guided by the fact that,  the solutions to such systems often exhibit complex and self-similar patterns over different scales, akin to fractal geometries,  we have explicitly represented the solutions of homogeneous system of $\alpha$-order linear differential equations with constant coefficients: $D^{\alpha}_F\bold{x}(t)=\bold{A}\bold{x}(t),$ illustrating everything with significant examples.\\
\textbf{Declaration of Competing Interest:}\\
The authors declare that they have no known competing financial interests or personal relationships that could have appeared to influence the work reported in this paper.\\

\section{References}
\bibliographystyle{elsarticle-num}
\bibliography{Refrancesmasystem1}






\end{document}